\documentclass[12pt]{article}
\usepackage{amsmath, amsfonts, amsthm, amssymb, amscd}
\usepackage[all]{xy}
\usepackage{setspace}
\usepackage{cite}
\usepackage[pdftex]{graphicx}
\usepackage{relsize}

\newtheorem{thm}{Theorem}[section]
\newtheorem{cor}[thm]{Corollary}
\newtheorem{lem}[thm]{Lemma}
\newtheorem{prop}[thm]{Proposition}

\newtheorem{que}[thm]{Question}
\newtheorem{rmk}[thm]{Remark}

\newcommand{\p}{\prime}

\newcommand{\R}{\mathbb{R}}
\newcommand{\Z}{\mathbb{Z}}

\newcommand{\C}{\mathbb{C}}

\newcommand{\Ha}{\mathbb{H}}

\newcommand{\con}{\equiv}

\newcommand{\HH}{\mathcal{H}}

\newcommand{\GG}{\mathcal{G}}

\newcommand{\Mod}[1]{\,(\operatorname{mod}\, #1)}

\newcommand{\inj}{\hookrightarrow}

\newcommand{\comment}[1]{}

\newcommand{\bi}{\begin{itemize}}
\newcommand{\ei}{\end{itemize}}
\newcommand{\ben}{\begin{enumerate}}
\newcommand{\een}{\end{enumerate}}
\newcommand{\be}{\begin{equation}}
\newcommand{\ee}{\end{equation}}
\newcommand{\bea}{\begin{eqnarray}}
\newcommand{\eea}{\end{eqnarray}}
\newcommand{\bal}{\begin{align}}
\newcommand{\eal}{\end{align}}
\newcommand{\ba}{\begin{array}}
\newcommand{\ea}{\end{array}}
\newcommand{\nn}{\nonumber}

\DeclareMathOperator{\GL}{GL}

\DeclareMathOperator{\SO}{SO}

\DeclareMathOperator{\Sk}{Sk}

\DeclareMathOperator{\Tr}{Tr}

\DeclareMathOperator{\Ind}{Ind}
\DeclareMathOperator{\Res}{Res}

\newcommand{\sS}{\mathcal{S}}
\newcommand{\tT}{\mathcal{T}}

\begin{document}
\title{Word-Induced Measures on Compact Groups}
\author{Gene S. Kopp and John D. Wiltshire-Gordon}
\maketitle
\begin{abstract}
Consider a group word $w$ in $n$ letters.  For a compact group $G$, $w$ induces a map $G^n \to G$ and thus a pushforward measure $\mu_w$ on $G$ from the Haar measure on $G^n$.  We associate to each word $w$ a $2$-dimensional cell complex $X(w)$ and prove in Theorem \ref{measures} that $\mu_w$ is determined by the topology of $X(w)$.  The proof makes use of non-abelian cohomology and Nielsen's classification of automorphisms of free groups \cite{nielsen}.  Focusing on the case when $X(w)$ is a surface, we rediscover representation-theoretic formulas for $\mu_w$ that were derived by Witten in the context of quantum gauge theory \cite{witten}.  These formulas generalize a result of Erd\H{o}s and Tur\'{a}n on the probability that two random elements of a finite group commute \cite{erdos}.  As another corollary, we give an elementary proof that the dimension of an irreducible complex representation of a finite group divides the order of the group; the only ingredients are Schur's lemma, basic counting, and a divisibility argument.
\end{abstract}

\section{Introduction}

How do we measure the degree to which a group is abelian?  Several parameters come to mind, each with a distinct flavor: 
\begin{itemize}
\item[(1)] the index of the derived subgroup;
\item[(2)] the average size of a conjugacy class;
\item[(3)] the dimensions of its largest irreducible representations; or
\item[(4)] the probability that two (uniformly chosen) elements commute.
\end{itemize}

\noindent So how abelian is the quaternion group $Q_8$?  Its derived subgroup has index $4$; its average conjugacy class has $8/5$ elements; it has irreducible representations of dimensions $1,1,1,1,$ and $2$.  Let's check the probability that two elements commute by calculating the commutator of every pair of elements:

\be
\begin{array}{r|cccccccc}
 [t_1,t_2] & 1 & i & j & k & -1 & -i & -j & -k \\
 \hline
 1 & 1 & 1 & 1 & 1 & 1 & 1 & 1 & 1 \\
 i & 1 & 1 & -1 & -1 & 1 & 1 & -1 & -1 \\
 j & 1 & -1 & 1 & -1 & 1 & -1 & 1 & -1 \\
 k & 1 & -1 & -1 & 1 & 1 & -1 & -1 & 1 \\
 -1 & 1 & 1 & 1 & 1 & 1 & 1 & 1 & 1 \\
 -i & 1 & 1 & -1 & -1 & 1 & 1 & -1 & -1 \\
 -j & 1 & -1 & 1 & -1 & 1 & -1 & 1 & -1 \\
 -k & 1 & -1 & -1 & 1 & 1 & -1 & -1 & 1
\end{array}
\ee
The table contains the identity element exactly 40 times, so the probability that two elements commute is $5/8$.  It is no coincidence that this number is the reciprocal of the average size of a conjugacy class, since Erd\H{o}s and Tur\'{a}n \cite{erdos} prove this to be true for finite groups in general.

More generally, we may wonder about group identities beyond commutation.  A \textbf{group word} $w$ on $n$ letters is an element of the free group $F_n$.  Given any $n$-tuple $\vec{g} \in G^n$ of elements in some group $G$, the universal property of $F_n$ provides an element called $w(\vec{g})$.  If $w(\vec{g})$ is the identity of $G$, we say that $w$ is \textbf{satisfied} at $t$.  With this terminology, we may ask a fundamental question of statistical group theory:
\begin{que} How frequently is a given word satisfied in a given group?
\end{que}

\noindent Here is notation for the quantity in question:

\be \gamma_G(w) = \# \left\{ t \in G^n \, \middle| \, w(t) = 1 \right\} . \ee
For example, we have already computed $\gamma_{Q_8} ([g_1,g_2]) = 40$.


Changing perspective somewhat, consider a compact group $G$ and a word $w \in F_n$.  We define a measure $\mu_w$ on $G$ as follows: for a Haar-measurable function $f : G \to \C$,
\be 
\int_G f d\mu_w := \int_{G^n} f(w(g_1, \ldots, g_n))dg_1 \cdots dg_n,
\ee
where the integral on the right is taken with respect to the normalized Haar measure on $G^n$.  In the finite case, $\mu_w$ is related to $\gamma_G(w)$ by
\be 
\int_G \delta_1 d\mu_w = |G|\mu_w(\{1\}) = |G|^{1-n} \gamma_G(w),
\ee
where $\delta_1$ denotes a point-mass at the identity of $G$.  

\section{Interpreting Group Words Topologically}

\noindent One way to understand group words is topological, since there's a natural way to build a cell complex out of a word.  If $w(a,b)=abca^{-1}b^{-1}c^{-1}$, for instance, we obtain the following space.

\begin{center} \includegraphics[scale=.5]{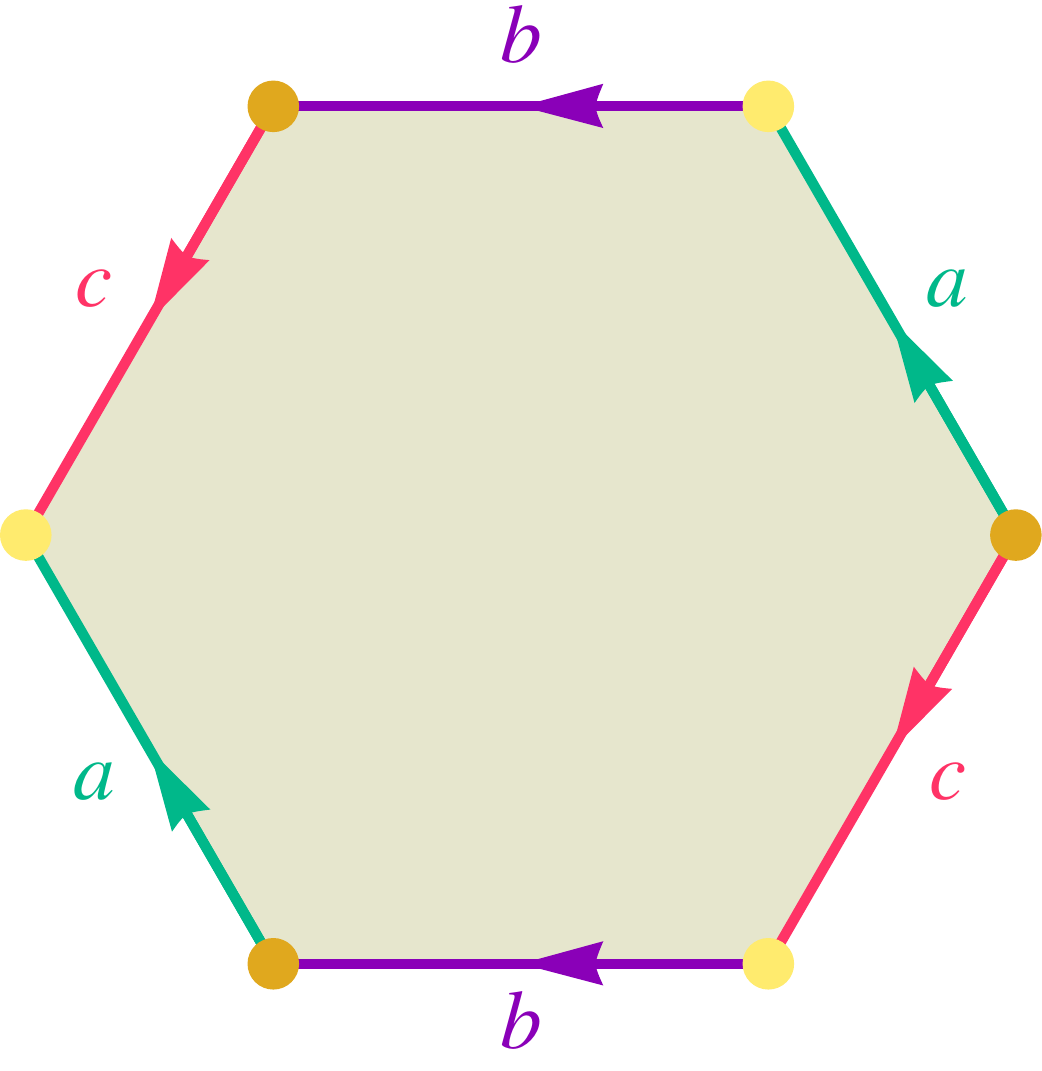} \end{center}

\noindent The space consists of a cycle built out of $w$ and a single $2$-cell attached along it, modulo the identifications of $1$-cells given by the word.  The attaching map from the boundary of the $2$-cell is called the \textbf{word polygon}.

We start with an easy topological lemma about such spaces:
\begin{lem}\label{skeleton}
Let $w_1$ and $w_2$ be words, and let $\delta_1$ and $\delta_2$ be the word polygons.  If $X(w_1)$ and $X(w_2)$ are homeomorphic, then there's a homotopy equivalence between the $1$-skeletons
\be 
h: \Sk_1 X(w_1) \to \Sk_1 X(w_2)
\ee
for which $h \mbox{\small{\,$\circ$\,}} \delta_1$ and $\delta_2$ are homotopic.
\end{lem}
\begin{proof}
Let $f : \xymatrix{X(w_1) \ar[r]^\sim & X(w_2)}$ be the homeomorphism.  Pick a point $x_2 \in X_2$ that doesn't lie in $h(\Sk_1 X_1) \cup \Sk_1 X_2$, and let $x_1 := f^{-1}(x_2)$.  Puncture each space at $x_i$; now $f$ restricts to a homeomorphism $\bar{f}$ between $X_1 \setminus \{x_1\}$ and $X_2 \setminus \{x_2\}$.  Each $X_i \setminus \{x_i\}$ has an easy deformation retraction to $\Sk_i X_i$, given by enlarging the hole.  The two $1$-skeletons are now deformation retracts of homeomorphic spaces, so we get a homotopy equivalence between them.  The word polygon $\delta_1$ is homotopic to a small circle around $x_1$, which is taken by $\bar{f}$ to a small circle around $x_2$; this new circle is homotopic to $\delta_2$.  Thus, we have a homotopy between $h \mbox{\small{\,$\circ$\,}} \delta_1$ and $\delta_2$.
\end{proof}

For example, when we puncture the cell complex we examined earlier, we see that the blue circle is homotopic to the word polygon.  
\begin{center} 
\includegraphics[scale=.4]{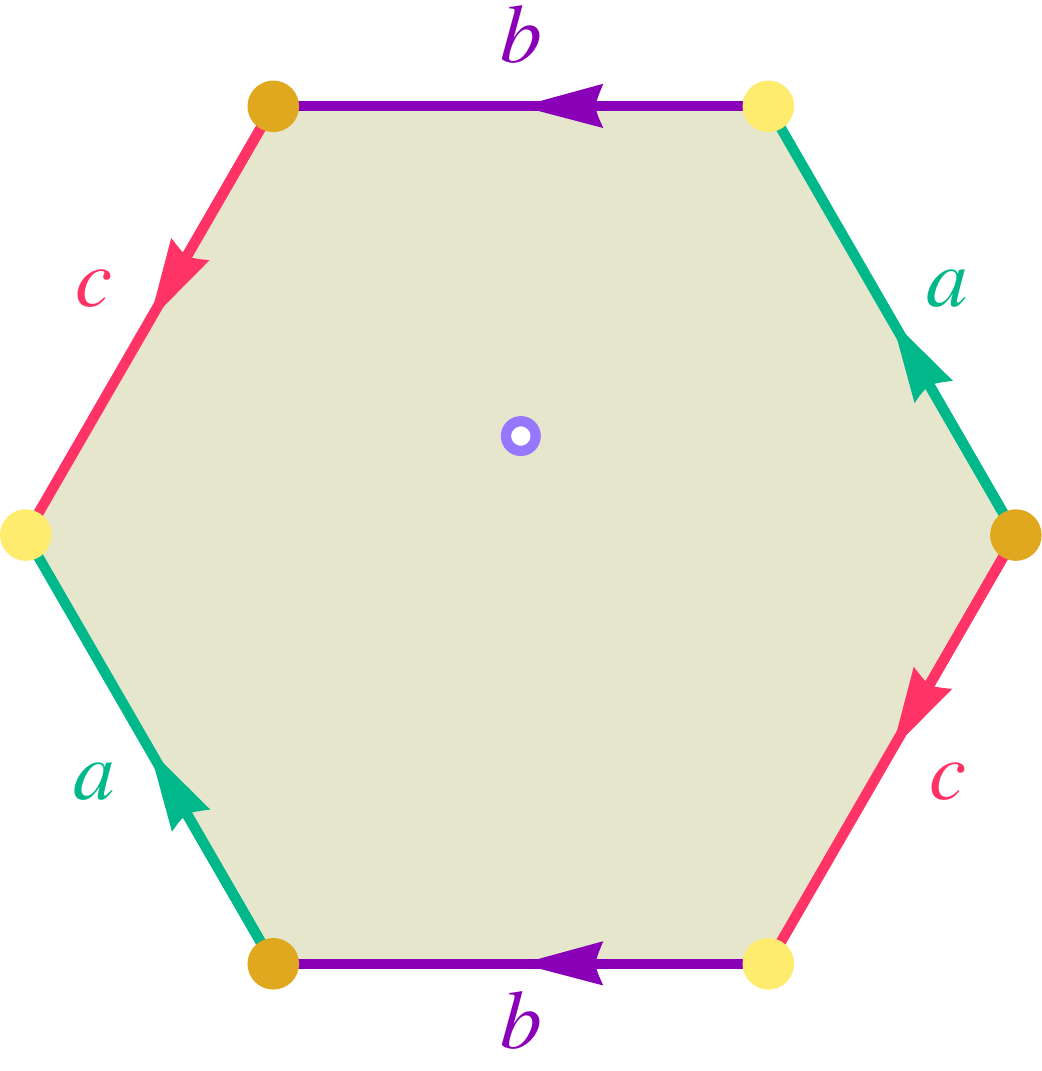} \ \ \ \ \ 
\includegraphics[scale=.4]{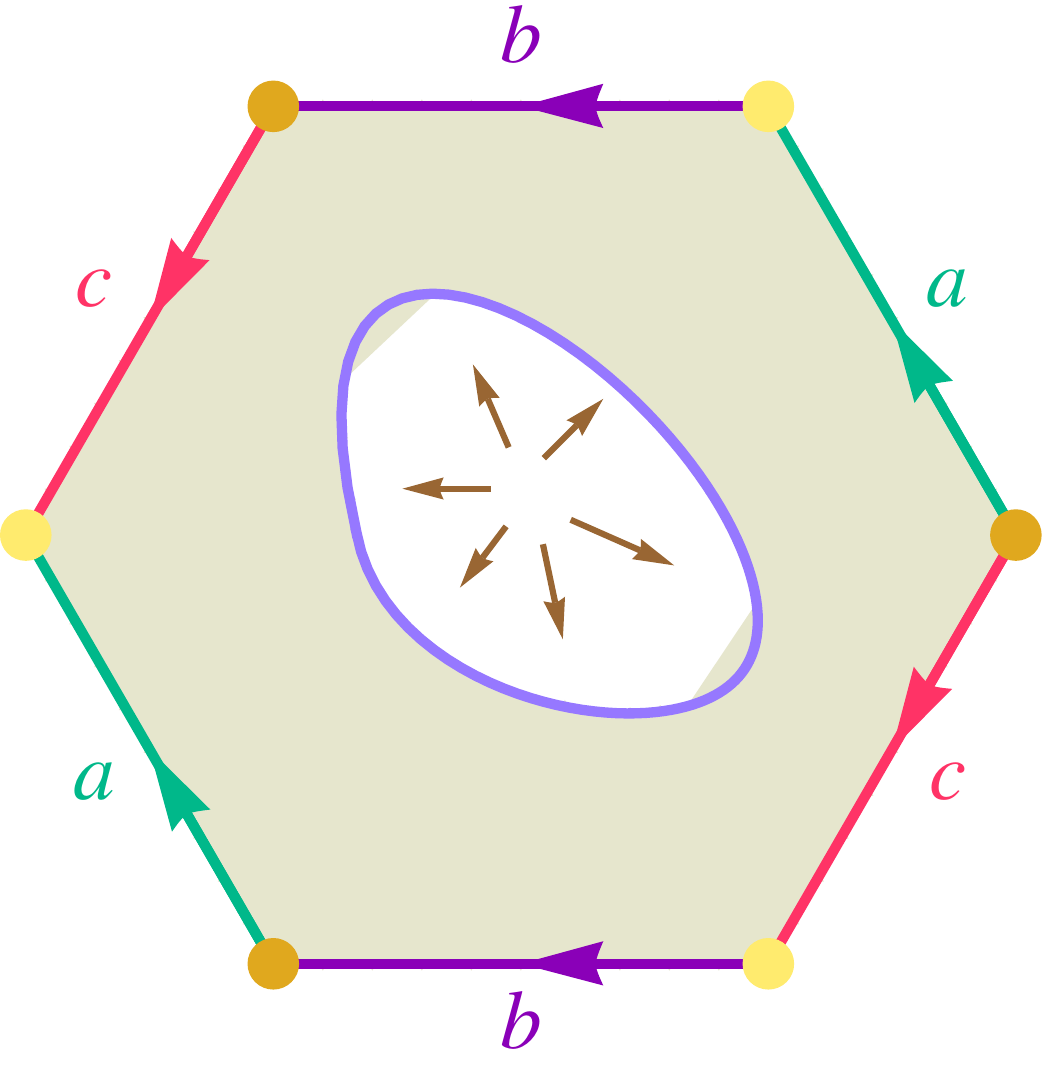} \ \ \ \ \ 
\includegraphics[scale=.4]{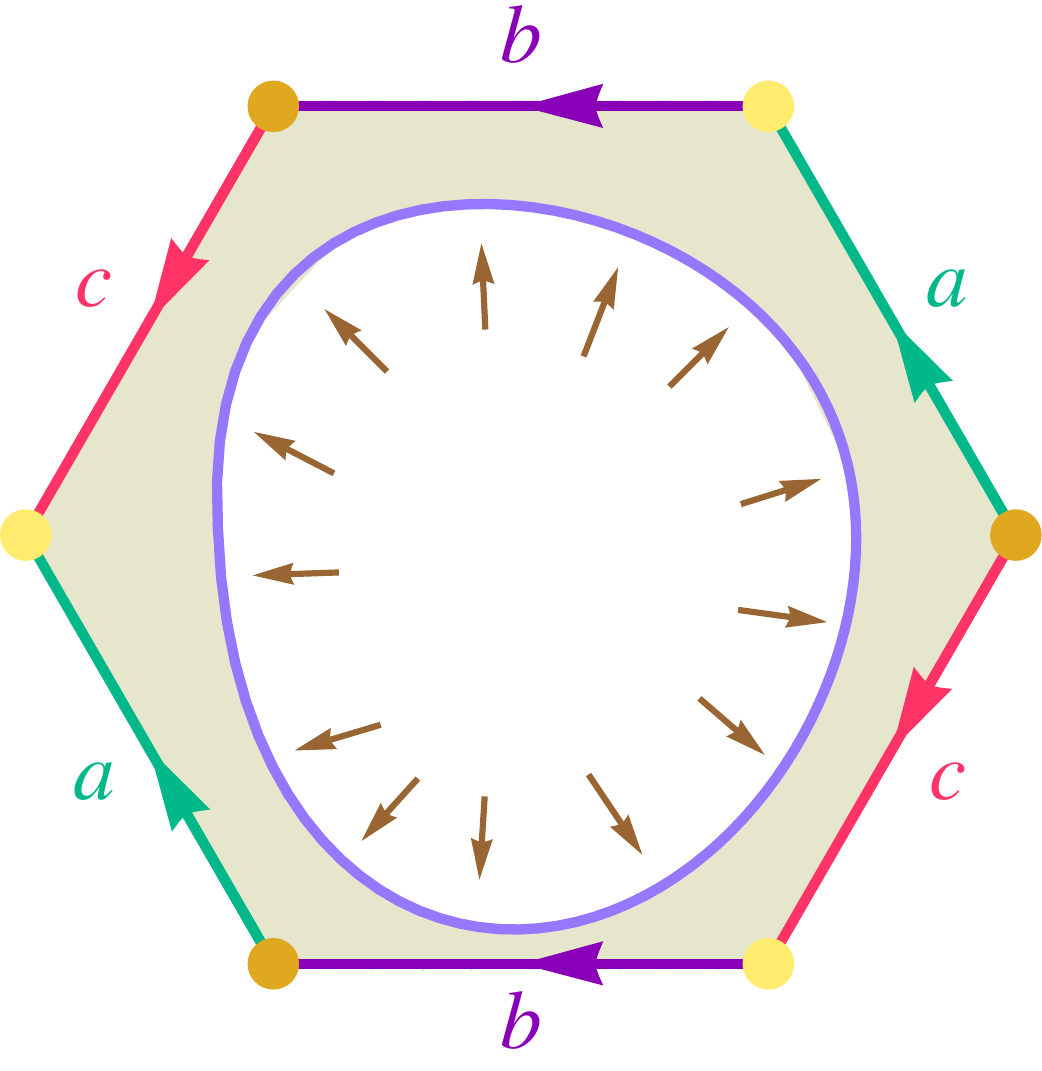} 
\end{center}

To continue this topological interpretation of group words, we must find a procedure analogous to evaluating a word at an $n$-tuple $\vec{g} \in G^n$.  To this end, we label each $1$-cell in $X(w)$ with the corresponding coordinate of $\vec{g}$, obtaining a cellular $1$-cocycle with coefficients in $G$.  The set of such $1$-cocycles inherits a natural measure structure from the normalized Haar measure on $G^n$.  We equip the set of cohomology classes with the natural quotient measure.

\begin{prop}\label{oneedge}
Let $\GG$ be a finite connected graph containing an edge $e_1$ between distinct vertices.  Let $\HH$ be quotient space obtained by modding out by $e_1$.
Then, there is a homotopy equivalence
\be 
\xymatrix{
\GG \ar@/^.5pc/[r]^r & \ar@/^.5pc/[l]^s \HH
}
\ee
so that the induced isomorphisms on non-abelian cohomology
\be 
\xymatrix{
H^1(\GG,G) \ar@/_.5pc/[r]_{s^*} & \ar@/_.5pc/[l]_{r^*} H^1(\HH,G)
}
\ee
are measure-preserving.
\end{prop}
\begin{proof}
Say $e_1$ is between vertices $U$ and $V$, with $U \neq V$.  Say $e_2,...,e_m$ are the other edges containing $V$, and $e_{m+1},...,e_n$ are the remaining edges.  In $\HH$, name the corresponding vertex $V^\p$ and edges $e_2^\p, \ldots, e_n^\p$.

Let $r : \GG \to \HH$ be the quotient map.  Let $s$ be a map taking $e_i^\p$ to $e_1e_i$ for $2 \leq i \leq m$ and to $e_i$ for $m+1 \leq i \leq n$.  It's easy to see that $sr$ and $rs$ are homotopy equivalent to the identity maps on $\GG$ and $\HH$, respectively.

\be 
\xymatrix{
\includegraphics[scale=.33]{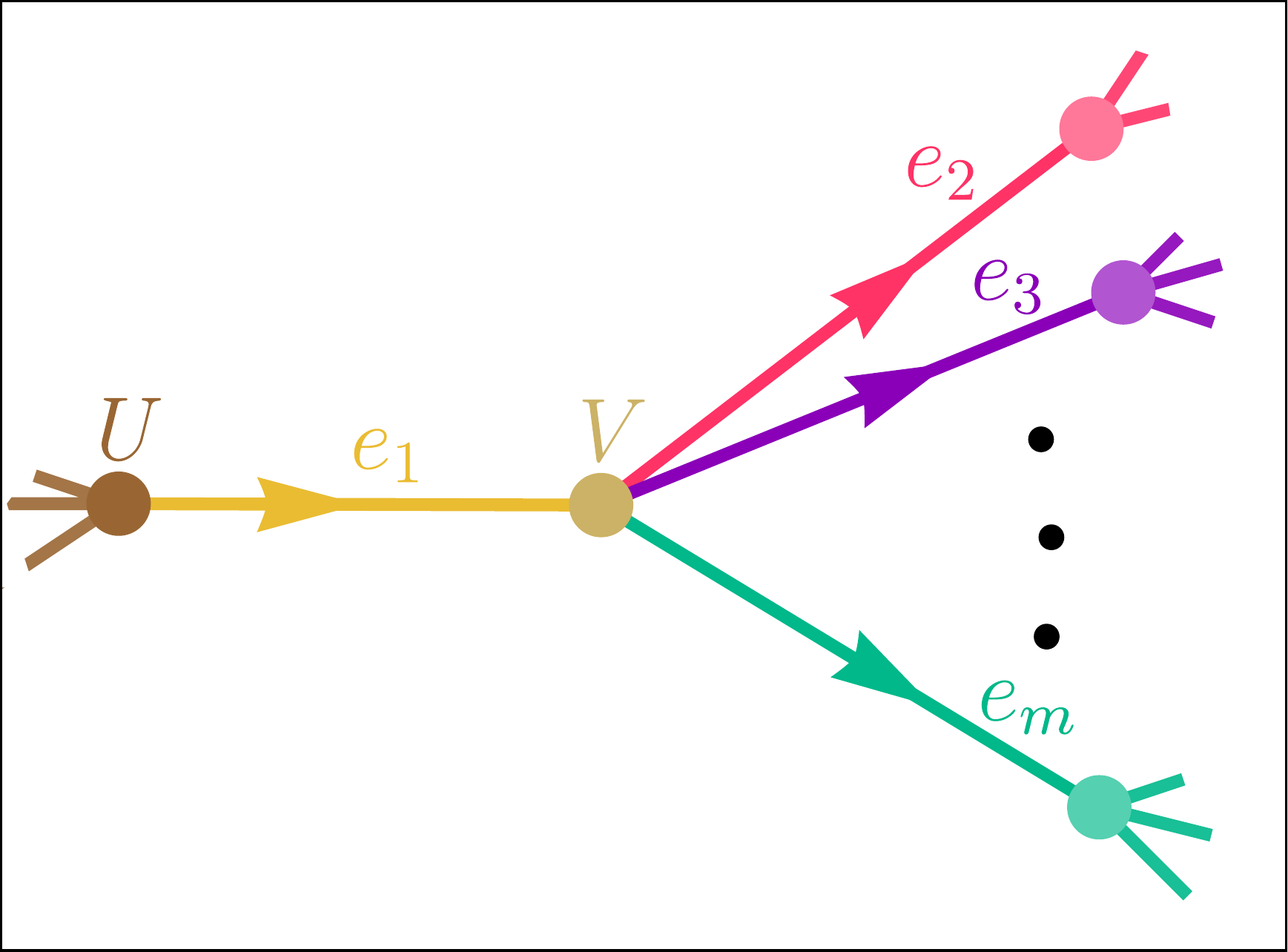} \ \ \ \ar[r]^r & \ \ \ \includegraphics[scale=.33]{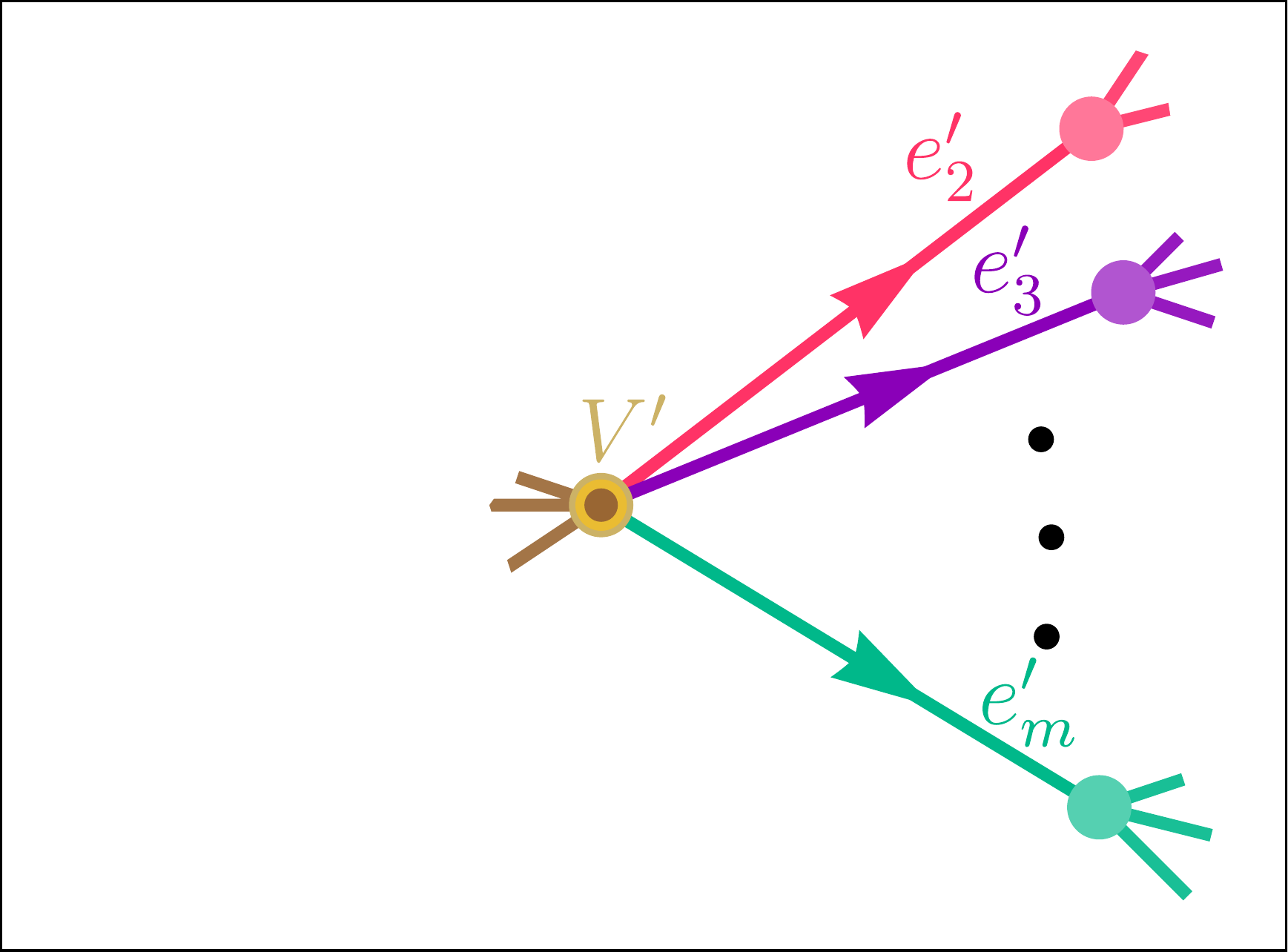}
}
\ee
\be
\xymatrix{
\includegraphics[scale=.33]{equivalence1.pdf} \ \ \ & \ \ \ \includegraphics[scale=.33]{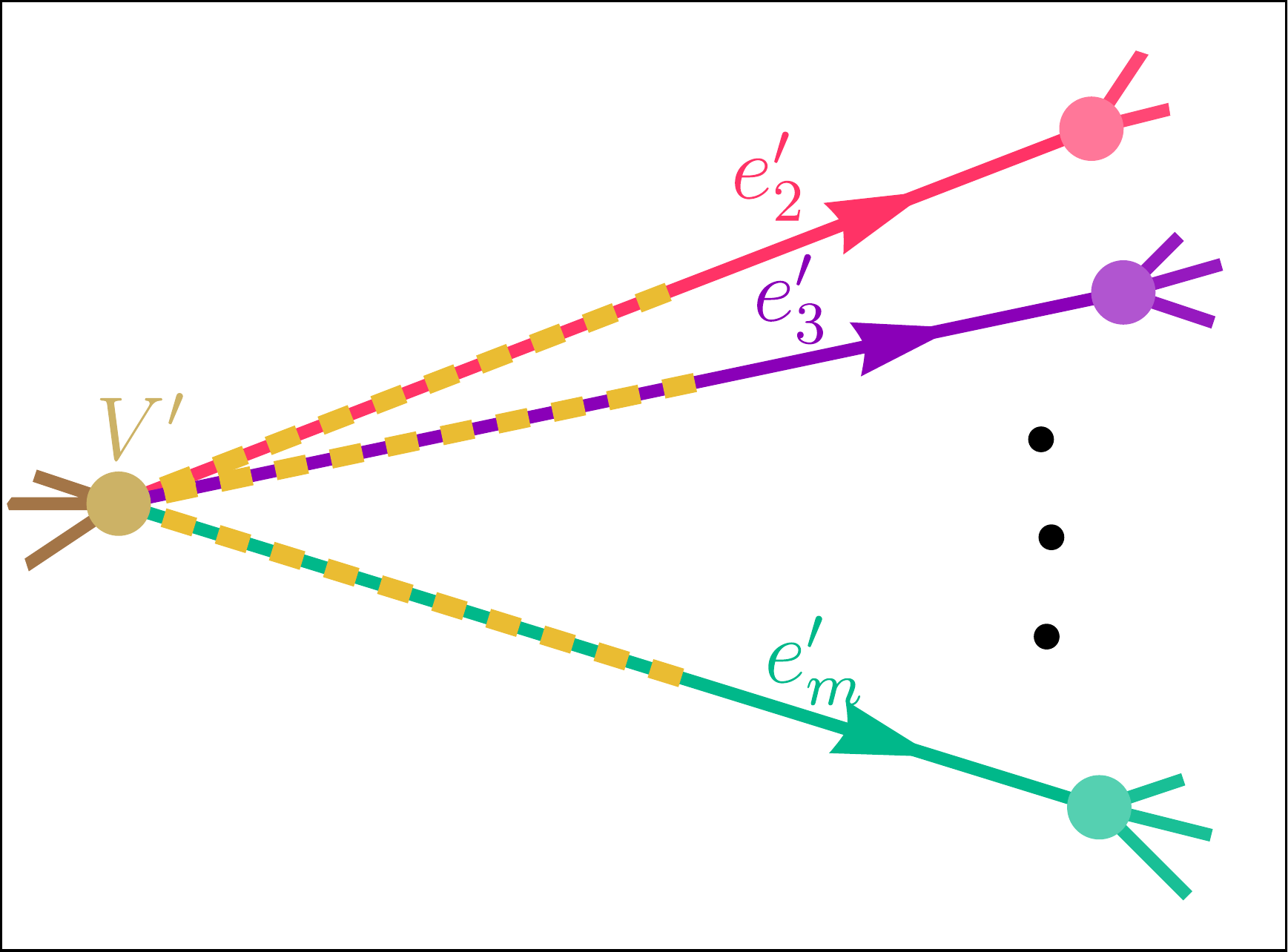} \ar[l]_s
}
\ee

Now we look at what these maps do to cocycles (with coefficients in $G$).  In order to write down cocycles explicitly in non-abelian cohomology, we must choose directions for the edges; for convenience, choose $e_1$ to culminate in $V$, $e_2,...,e_m$ to emanate from $V$, and $e_2^\p,...,e_m^\p$ to emanate from $V^\p$.  For the remaining edges, put the same (arbitrary) direction on $e_i$ and $e_i^\p$.  We then have
\bea
r^*(g_2,...,g_n) &=& (1,g_2,...,g_n)\nn\\
s^*(g_1,...,g_n) &=& (g_1g_2,...,g_1g_m,g_{m+1},...,g_n).
\eea

Say $f : H^1(\GG, G) \to \C$ is a measurable function.  In $H^1(\GG, G)$, the cocycles $(g_1,...,g_p)$ and $(1,g_1^{-1}g_2,...,g_1^{-1}g_m,g_{m+1},...,g_n)$ are cohomologous.  We have
\bea
\int_{G^n} f(g_1,...,g_p)d\vec{g}
&=& \int_{G^n} f(1,g_1^{-1}g_2,...,g_1^{-1}g_m,g_{m+1},...,g_n)d\vec{g}\nn\\
&=& \int_G \int_{G^{n-1}} f(1,g_2',...,g_n')d\vec{g'}dg_1\nn\\
&=& \int_{G^{n-1}} f(1,g_2',...,g_n')d\vec{g'}\nn\\
&=& \int_{G^{n-1}} f(r^*(g_2',...,g_n'))d\vec{g'},
\eea
provided one of the integrals converges.  Similarly,
\bea
\int_{G^n} f(s^*(g_1,...,g_p))d\vec{g}
&=& \int_{G^n} f(g_1g_2,...,g_1g_m,g_{m+1},...,g_n)d\vec{g}\nn\\
&=& \int_G \int_{G^{n-1}} f(g_2',...,g_n')d\vec{g'}dg_1\nn\\
&=& \int_{G^{n-1}} f(g_2',...,g_n')d\vec{g'}.
\eea
Thus, $r^*$ and $s^*$ are measure-preserving.
\end{proof}

\begin{cor}\label{littlecrap}
Let $\GG$ be a finite connected graph.  Then, there is a homotopy equivalence
\be 
\xymatrix{
\GG \ar@/^.5pc/[r]^r & \ar@/^.5pc/[l]^s \HH
}
\ee
so that $\HH$ has only one vertex and the induced isomorphisms on non-abelian cohomology
\be 
\xymatrix{
H^1(\GG,G) \ar@/_.5pc/[r]_{s^*} & \ar@/_.5pc/[l]_{r^*} H^1(\HH,G)
}
\ee
are measure-preserving.
\end{cor}
\begin{proof}
Choose a spanning tree for $\GG$ and apply Proposition \ref{oneedge} repeatedly until every edge in the spanning tree has been contracted.
\end{proof}

\begin{lem}\label{measurepreserving}
For any homotopy equivalence of finite connected graphs 
\be 
\xymatrix{
\GG_1 \ar[r]^h & \GG_2
},
\ee
the induced isomorphism on non-abelian cohomology
\be 
\xymatrix{
H^1(\GG_1,G) & \ar[l]_{h^*} H^1(\GG_2,G)
}
\ee
is in fact measure-preserving.
\end{lem}
\begin{proof}
By Lemma \ref{littlecrap}, we can find equivalences
\be 
\xymatrix{
\GG_1 \ar[r]^h & \GG_2 \ar[d]^r\\
\HH_1 \ar[u]^s & \HH_2
}
\ee
so that $\HH_1$ and $\HH_2$ have only one vertex.  The composition $rhs$ is a homotopy equivalence between $\HH_1$ and $\HH_2$.  It defines an isomorphism between their fundamental groups, both of which are the free group on the edges (and so they both have the same number of edges).  Choosing an arbitrary bijection between the edge sets; this map becomes an automorphism of $F_n$.  All automorphisms of $F_n$ are generated by elementary Nielsen transformations on the generating set: permuting the generators, replacing a generator by its inverse, and ``shearing'' a generator by multiplying it by another generator \cite{nielsen, nielsenenglish}.  The corresponding substitutions on $H^1(\HH_1,G) = G^n = H^1(\HH_2,G)$ are valid Haar substitutions.  Thus, in the following diagram, (all the maps are isomorphism and) $(rhs)^* = s^*h^*r^*$, $r^*$, and $s^*$ are measure-preserving.
\be 
\xymatrix{
H^1(\GG_1,G) \ar[d]_{s^*} & H^1(\GG_2,G) \ar[l]_{h^*}\\
H^1(\HH_1,G) & H^1(\HH_2,G) \ar[u]_{r^*}
}
\ee
It follows that $h^*$ is also measure-preserving.
\end{proof}

\begin{thm}\label{measures}
Let $w_1$ and $w_2$ be words giving rise to homeomorphic cell complexes $X(w_i)$, and let $G$ be a compact Hausdorff group.  If $f : G \to \C$ is a measurable function on the group, then
\be \int \limits_{G^{n_1}} f(w_1(\vec{g})) d\vec{g} = \int \limits_{G^{n_2}} f(w_2(\vec{g})) d\vec{g}, \ee
provided one of the integrals converges.  In other words, the measures $\mu_{w_1} = \mu_{w_2}$.
\end{thm}
\begin{proof}
By Lemma \ref{skeleton}, there is a homotopy equivalence between the $1$-skeletons of $X_1:=X(w_1)$ and $X_2:=X(w_2)$:
\be 
h: \Sk_1 X_1 \to \Sk_1 X_2
\ee
under which the word polygons correspond.
By Lemma \ref{measurepreserving}, the induced map
\be 
h^*: H^1(\Sk_1 X_1,G) \to H^1(\Sk_1 X_2,G)
\ee
is measure-preserving.  The word map $w_1 : G^{n_1} \to G$ is constant on cohomological equivalence classes because the word polygon is a cycle.  Therefore,
\be 
\int_{G^{n_1}} f(w_1(\vec{g}))d\vec{g} = \int_{G^{n_2}} f(w_1(h^*(\vec{g})))d\vec{g}.
\ee
By the definition of the maps, $f(w_1(h^*(\vec{g}))) = f((h_* w_1)(\vec{g}))$, where
\be 
h_* : \Pi_1(\Sk_1 X_1) \to \Pi_1(\Sk_2 X_1)
\ee
denotes the induced equivalence of fundamental groupoids.  Because the word polygons correspond under $h$, we have $h_* w_1 = w_2$.  Putting it all together,
\be 
\int_{G^{n_1}} f(w_1(\vec{g}))d\vec{g} = \int_{G^{n_2}} f(w_2(\vec{g}))d\vec{g}.
\ee
\end{proof}


\section{A Representation-Theoretic Perspective}

If $\rho$ is a (continuous) irreducible representation of a compact group $G$, then the $\R$-algebra generated by the matrices in $\rho(G)$ is simple.  By the Artin-Wedderburn theorem, it is isomorphic to a matrix algebra over a division algebra over $\R$.  The division algebras over $\R$ are precisely the real field $\R$, the complex field $\C$, and the quaternion skew field $\Ha$.  We say that $\rho$ is either real, complex, or quaternionic.  The so-called Frobenius-Schur indicator provides an easy way to tell them apart:

\begin{thm}[Frobenius, Schur]
Let $G$ be a compact group equipped with its normalized Haar measure, and let $\rho : G \to \GL(V)$ be an irreducible representation of $G$. Let $\chi^\rho$ be the character of $\rho$.  Now
\be 
v(\chi^\rho) := \int \limits_G \chi^\rho(g^2) dg  = \left\{
     \begin{array}{ll}
       1  & : \rho \mbox{ is real} \\
       0  & : \rho \mbox{ is complex} \\
       -1 & : \rho \mbox{ is quaternionic}
     \end{array};
   \right.
\ee
this quantity is called the Frobenius-Schur indicator.
\end{thm}
In the topological interpretation from before, the word $g^2$ corresponds to a projective plane.  In fact, we are now prepared to begin connecting the two perspectives.

In general, for a group word $w$ in $n$ letters, define the linear functional
\be 
v_w(f) := \int_G f d\mu_w.
\ee
The Frobenius-Schur indicator corresponds to the case of $w(g) = g^2$.

The functional $v_w$, and thus the measure $\mu_w$, is completely determined its value at the irreducible characters:
\be\label{determined} 
v_w(f) = \sum_{\chi \in \hat{G}} \langle f, \chi \rangle v_w(\chi).
\ee
We do not require that $f$ be a class function because $\mu_w$ is uniform on conjugacy classes.

It turns out that the ``indicators'' $\mu_w(\chi)$ already tell use what happens when we integrate the actual matrices in our representation, entrywise, without taking traces.
\begin{lem}\label{wordconj}
Let $\rho : G \to GL(V)$ be an irreducible representation of a compact group $G$.  Then,
\be 
\int_{G^n} \rho(w(\vec{g}))d\vec{g} = \frac{v_w(\chi^\rho)}{\dim{V}} I.
\ee
\end{lem}
\begin{proof}
Set
\be 
A := \int_{G^n} \rho(w(\vec{g}))d\vec{g}.
\ee
A simple change of variable show that the action of $A$ commutes with the action of $G$:
\bea
A \rho(h) &=& \int_{G^n} \rho(w(\vec{g}))\rho(h)d\vec{g}\nn\\
          &=& \int_{G^n} \rho(w(\vec{g})h)d\vec{g}\nn\\
          &=& \int_{G^n} \rho(hw(\vec{h^{-1}gh}))d\vec{g}\nn\\
          &=& \rho(h) \int_{G^n} \rho(w(\vec{h^{-1}gh}))d\vec{g}\nn\\
          &=& \rho(h) \int_{G^n} \rho(w(\vec{g'}))d\vec{g'}\nn\\
          &=& \rho(h) A,
\eea
where we have made the substitutions $g_i' = h^{-1}g_ih$ and have used Haar invariance.  We conclude that $A$ must be an endomorphism of the simple $L^2 G$-module $V$.  By Schur's lemma, $A=\lambda I$, some multiple of the identity.  We may compute $\lambda = \Tr(A)/\dim{V} = v_w(\chi^\rho)/\dim{V}$.
\end{proof}

A similar argument gives us the following:
\begin{lem}\label{wordindicator}
If $G$ is a compact group and $\rho : G \to GL(V)$ is an irreducible representation, then
\be 
\int \limits_G \rho(x g x^{-1}) dx = \frac{\chi^{\rho}(g)}{\dim{V}} I.
\ee
\end{lem}
\begin{proof}
For $g \in G$, define 
\be 
I_g = \int \limits_G \rho(x g x^{-1}) dx.
\ee
It is easy to see that the action of $I_g$ commutes with the action of $G$:
\bea I_g \rho(h) & = & \int \limits_G \rho(x g x^{-1} h) dx \nn\\
 & = &  \int \limits_G \rho(h y g y^{-1}) dy \nn\\
 & = & \rho(h) I_g,
\eea
where we have made the substitution $y = h^{-1} x$ and have used Haar invariance.  We conclude that $I_g$ must be an endomorphism of the simple $L^2 G$-module $V$.  By Schur's lemma, $I_g=\lambda I$, some multiple of the identity.  We may compute $\lambda$ explicitly:
\bea \lambda & = & \frac{\Tr (\lambda I)}{\dim{V}} \nn\\
 & = & \frac{\Tr(I_g)}{\dim{V}} \nn\\
 & = & \frac{\int_G \chi^{\rho}(x g x^{-1}) dx}{\dim{V}} \nn\\
 & = & \frac{\int_G \chi^{\rho}(g) dx}{\dim{V}} \nn\\
 & = & \frac{ \chi^{\rho}(g)}{\dim{V}} 
\eea
\end{proof}

\section{Surfaces and Witten Zeta functions}
\begin{thm}\label{surfaceintegrals}
Let $w$ be a word in $n$ letters that defines a surface (that is, a topological $2$-manifold), and let $\rho : G \to \GL(V)$ be an irreducible representation of some compact group $G$.  Now if $X(w)$ is orientable,
\be \int \limits_{G^n} \rho (w(t)) dt = (\dim V)^{\kappa - 2}I; \ee
if not,
\be \int \limits_{G^n} \rho (w(t)) dt = v(\chi^\rho)^{2 - \kappa} (\dim V)^{\kappa - 2}I, \ee
where $\kappa$ denotes the Euler characteristic of $X(w)$.
\end{thm}
\begin{proof}
We start with the case of the torus $T=S^1 \times S^1$:
\bea
\int \limits_{G^2} \chi^{\rho} (t_1 t_2 t_1^{-1} t_2^{-1}) dt_1 dt_2 & = & \Tr \int \limits_{G^2} \rho (t_1 t_2 t_1^{-1} t_2^{-1}) dt_1 dt_2 \nn\\
 & = & \Tr \int \limits_{G} \left( \int \limits_{G} \rho (t_1 t_2 t_1^{-1}) dt_1 \right) \rho(t_2^{-1}) dt_2 \nn\\
 & = & \Tr \int \limits_{G} I_{t_2} \rho(t_2^{-1}) dt_2 \nn\\
 & = & \frac{1}{\dim V} \Tr \int \limits_{G} \chi^{\rho}(t_2) \rho(t_2^{-1}) dt_2 \nn\\
 & = & \frac{1}{\dim V} \int \limits_{G} \chi^{\rho}(t_2) \overline{\chi^{\rho}(t_2)} dt_2 \nn\\
 & = & \frac{\langle \chi^{\rho}, \chi^{\rho} \rangle}{\dim V} \nn\\
 & = & \frac{1}{\dim V}.
\eea

By Lemma \ref{wordconj}, we in fact have
\be \int \limits_{G^2} \rho (t_1 t_2 t_1^{-1} t_2^{-1}) dt_1 dt_2 = \frac{1}{(\dim V)^2}I. \ee


The classification of surfaces implies that an orientable $X(w)$ can be written as a connected sum of tori:
\be X(w) = T \# T \# \cdots \# T = T^{\# k}. \ee
To form the connected sum of two words, we simply multiply them in the free product of their respective free groups.

Earlier propositions allow us to verify the claim for a single word representing each genus.  Let us represent the genus $\gamma$ surface by the word
\be w(t_1, t_2, \ldots, t_{2\gamma}) = [t_1,t_2][t_3,t_4] \cdots [t_{2\gamma-1},t_{2\gamma}]. \ee
Now we split the $T^{\# \gamma}$ case into $\gamma$ copies of the $T$ case:
\bea 
\int \limits_{G^{2\gamma}} \rho([t_1,t_2][t_3,t_4] \cdots [t_{2\gamma-1},t_{2\gamma}]) dt & = & \int \limits_{G^{2\gamma}} \rho([t_1,t_2])\rho([t_3,t_4]) \cdots \rho([t_{2\gamma-1},t_{2\gamma}]) dt \\
& = & \prod_{j=1}^{\gamma} \int \limits_{G} \int \limits_{G} \rho([t_{2j-1},t_{2j}]) dt_{2j-1} dt_{2j}\\
& = & \prod_{j=1}^{\gamma} (\dim V)^{-2} \\
& = & (\dim V)^{-2\gamma}
\eea
A non-orientable surface is given by a connected sum of the form
\be X(w) = P \# T^{\# \gamma} \mbox{           or           } X(w) = P \# P \# T^{\# \gamma}, \ee
where $P$ is a projective plane.  It is easy to see, however, that the theorem of Frobenius and Schur allows us to apply the same technique in these cases to obtain the required result.
\end{proof}

Putting this together with earlier calculations (\ref{determined}), we have
\begin{cor}\label{functional}
For a word $w$ with $X(w)$ a surface of Euler characteristic $\kappa$, and any measurable $f: G \to \C$,
\be 
v_w(f) = \sum_{\rho \in \hat{G}} \langle f, \chi^\rho \rangle (\dim \rho)^{\kappa - 1}
\ee
if $X(w)$ is orientable, and 
\be 
v_w(f) = \sum_{\rho \in \hat{G}} \langle f, \chi^\rho \rangle v(\chi^\rho)^{2 - \kappa} (\dim \rho)^{\kappa - 1}
\ee
if $X(w)$ is non-orientable, so long as the sum converges.
\end{cor}

Similar formulas were derived by Witten for compact Lie groups in the context of quantum gauge theory \cite{witten}.  Witten goes on to define what is now known as the Witten zeta function of a compact group:
\be \zeta_{G}(s) = \sum_{\rho \in \hat{G}} (\dim \rho)^{-s}. \ee
Special values of this zeta function arise if we plug in the Dirac delta function---that is, the character of the regular representation---for $f$ in the orientable case of Corollary \ref{functional}:
\bea 
v_w(\delta_1) &=& \sum_{\rho \in \hat{G}} \langle \delta_1, \chi^\rho \rangle (\dim \rho)^{\kappa - 1}\nn\\
              &=& \sum_{\rho \in \hat{G}} \chi^\rho(1) (\dim \rho)^{\kappa - 1}\nn\\
              &=& \sum_{\rho \in \hat{G}} (\dim \rho)^{\kappa}\nn\\
              &=& \zeta_G(-\kappa).
\eea
This doesn't make sense beyond the case when $G$ is finite.  
This is fixed by going backwards: If $\sum_{\rho \in \hat{G}} (\dim \rho)^\kappa$ converges, the function
\be 
\zeta_G(-\kappa; g) := \sum_{\rho \in \hat{G}} \chi^\rho(g) (\dim \rho)^\kappa
\ee
is (as a function of $g$) a Radon-Nikodym derivative for $\mu_w$ with respect to $\mu$, and it is continuous.  So it is the unique\footnote{Consider two continuous Radon-Nikodym derivatives.  They are equal almost everywhere.  But, if they were unequal at any point, they would differ at an open set about that point, and nonempty open sets in $G$ have positive measure.  So they're equal.} continuous Radon-Nikodym derivative $\mu_w/\mu$.  In the language of distribution theory, $\zeta_G(\kappa; g)$ is the limit of $v_w(f_j)$ over any sequence $\langle f_j \rangle$ of $L^\infty$-functions approximating a point of mass one at $g$.  In particular, at the identity, we have
\be 
(\mu_w/\mu)(1) = \zeta_G(-\kappa).
\ee
In the finite case, we write this formula combinatorially:
\be\label{combinatron}
\gamma_G(w) = |G|^{n-1} v_w(\delta_1) = |G|^{n-1} \zeta_G(-\kappa).
\ee

\section{Statistical Group Theory}
Returning to our original example of counting commuting pairs in $Q_8$, we need only recall that the dimensions of the irreducible representations are $\{1, 1, 1, 1, 2\}$.  The Witten zeta function can be written explicitly:
\be \zeta_{Q_3}(s) = 1^{-s} + 1^{-s} + 1^{-s} + 1^{-s} + 2^{-s}. \ee
We recover not only the result that there are $40$ ways to satisfy the commutator word
\be |Q_8| \zeta_{Q_8}(0) = 8(1^0 + 1^0 + 1^0 + 1^0 + 2^{0}) = 40, \ee
but also that a product of three separate commutators can be satisfied in
\be |Q_8|^5 \zeta_{Q_8}(4) = 8^5(1^{-4} +  1^{-4} + 1^{-4} + 1^{-4} + 2^{-4}) = 133120 \ee
different ways.  

In the introduction, we mentioned several other ways to check the degree to which $Q_8$ is abelian.  It is easy to see that these quantities may also be calculated in terms of the Witten zeta function:
\begin{itemize}
\item[(1)] the index of the derived subgroup: \\
\be 
\lim_{s \rightarrow \infty } \zeta_{Q_8} (s) = 4
\ee
\item[(2)] the average size of a conjugacy class:
\be 
\frac{|Q_8|}{\zeta_{Q_8}(0)} = \frac{8}{5}
\ee
\item[(3)] the dimensions of its largest irreducible representations:
\be
D := \mbox{largest dimension occurring} =  \exp\left(-\lim_{s \rightarrow -\infty} \log (\zeta_{Q_8} (s))/s \right)=2
\ee
\be 
\mbox{number of distinct $D$-dimensional irreducibles } = \lim_{s \rightarrow -\infty} \zeta_{Q_8}(s)/D^{-s}=1.
\ee
\end{itemize}

We now give an application of the theorem in the infinite case.

\begin{cor} Let $G=\SO(3)$.  The word $w(t_1,t_2,t_3,t_4) = [t_1,t_2][t_3,t_4]$ provides a measure-preserving map from the probability space $G^4$ to $G$.  The resulting measure on $G$ is absolutely continuous with respect to Haar measure, and the continuous Radon-Nikodym derivative at the identity element is
\be \frac{\pi^2}{8}. \ee
\end{cor}
\begin{proof}
The rotation group $\SO(3)$ has unique irreducible representations of every odd dimension, and none of even dimension \cite{bro}.  Thus,
\bea 
\zeta_{\SO(3)}(s) &=& \sum_{n \mbox{\begin{scriptsize} odd\end{scriptsize}}} n^{-s}\nn\\
                  &=& \sum_{n=1}^\infty n^{-s} - \sum_{n=1}^\infty (2n)^{-s}\nn\\
                  &=& (1-2^{-s})\zeta(s).
\eea
So
\be 
\zeta_{\SO(3)}(2) = (1-2^{-2})\zeta(2) = \frac{3}{4} \cdot \frac{\pi^2}{6} = \frac{\pi^2}{8}.
\ee
\end{proof}
\section{A Corollary for Representations of Finite Groups}
Here we prove the advertised corollary that, for a finite group $G$, the dimension of an irreducible $\mathbb{C}$-representation of $G$ divides its order.  We begin with some easy number theory.
\begin{lem}\label{divisorsold}
Let $a_1, a_2, \ldots, a_n$ be integers.  In order to show that $d$ divides each $a_i$, it suffices to verify that
\be\label{duh}
a_1^k + a_2^k + \cdots + a_n^k \con 0 \Mod{d^k}
\ee
for all positive integers $k$.
\end{lem}
\begin{proof}
Suppose first that $d$ is a prime.  Let $m$ denote the number of $a_i$ not divisible by $d$; we seek to show that $m=0$.  Setting $k=\varphi(d^n)$, 
\be
a_1^{\varphi (d^n)} + a_2^{\varphi (d^n)} + \cdots+  a_n^{\varphi (d^n)} \con 0 \Mod{d^{\varphi (d^n)}}.
\ee
Since $\varphi (d^n) = d^n - d^{n-1} \geq n$,
\be
a_1^{\varphi (d^n)} + a_2^{\varphi (d^n)} + \cdots+  a_n^{\varphi (d^n)} \con 0 \Mod{d^{n}}.
\ee
We now apply Euler's totient theorem to each term on the left; those $a_i$ which are prime to $p$ contribute $1$ while the rest are $0$.  Thus,
\be
m \con 0 \Mod{d^{n}}.
\ee
However, $m \in \{0,1,\ldots, n\}$ and $n < d^n$, so $m=0$.

Consider now the case in which $d$ is a power of some prime $p$.  The condition
\be
a_1^k + a_2^k + \cdots + a_n^k \con 0 \Mod{d^k}
\ee
for all positive integers $k$ implies
\be
a_1^k + a_2^k + \cdots + a_n^k \con 0 \Mod{p^k}
\ee
and so the earlier argument gives us that $p$ divides each $a_i$.  Dividing through by $p$, we may iterate this argument to obtain that $d$ divides each $a_i$.

The general statement of the lemma now follows from the Chinese remainder theorem.
\end{proof}
\begin{cor}\label{divisors}
Let $q_1, q_2, \ldots , q_n$ be rational numbers.  Suppose that
\be
q_1^k + q_2^k + \cdots +  q_n^k \in \mathbb{Z}
\ee
for each positive integer $k$; it follows that each $q_i$ is an integer.
\end{cor}
\begin{proof}
Clear denominators and apply Lemma \ref{divisorsold}.
\end{proof}

\begin{thm}\label{classical}
If $G$ is a finite group, and $V$ is an irreducible representation of $G$ over $\C$, then $\dim V$ divides $|G|$.
\end{thm}
\begin{proof}
We apply (\ref{combinatron}) to an orientable surface of genus $k+1$:
\bea
\gamma_G(w_{k+1}) &=& |G|^{2k+1} \zeta_G(2k)\nn\\
                   &=& |G|^{2k+1} \sum_{\rho \in \hat{G}} \frac{1}{(\dim \rho)^{2k}}\nn\\
                   &=& |G| \sum_{\rho \in \hat{G}} \frac{|G|^{2k}}{(\dim \rho)^{2k}}.
\eea
For convenience, define $r_i$ to be $\frac{|G|^2}{(\dim \rho_i)^2}$.  We now observe:
\be 
\gamma_G(w_{k+1}) = \underbrace{r_1^k + r_1^k + \cdots + r_1^k}_{|G|} \,\, + \,\, \underbrace{r_2^k + r_2^k + \cdots + r_2^k}_{|G|} \,\, + \,\, \cdots \,\, + \,\, \underbrace{r_{n}^k + r_n^k + \cdots + r_n^k}_{|G|}
\ee
Since the left side is always an integer, Corollary \ref{divisors} gives that $r_i \in \mathbb{Z}$, and hence the dimension of any irreducible representation of $G$ divides $|G|$.
\end{proof}

\begin{rmk}
Theorem \ref{classical} relies only on Lemma \ref{wordconj}, Lemma \ref{wordindicator}, Corollary \ref{functional}, Lemma \ref{divisorsold}, and Corollary \ref{divisors}, the proofs of which are completely elementary if we restrict attention to finite groups and words that are products of independent commutators.  In particular, it does not rely on facts about algebraic integers, Nielsen's classification of automorphisms of free groups, non-abelian cohomology, Haar measure, Radon-Nikodym derivatives, or the homotopy category.
\end{rmk}

\bibliographystyle{alpha}{}
\bibliography{references}

\end{document}